\documentclass[a4paper,UKenglish,fullpage]{article}

\usepackage[top=2.54cm, bottom=2.54cm, left=2.54cm, right=2.54cm]{geometry}

\usepackage{microtype}

\usepackage{enumitem}
\usepackage{amsmath}  
\usepackage{graphicx}
\usepackage{amsfonts}
\usepackage{amsthm}

\newtheorem{theorem}{Theorem}
\newtheorem{definition}[theorem]{Definition}
\newtheorem{lemma}[theorem]{Lemma}

\graphicspath{{./figures/}}

\DeclareMathOperator{\crn}{cr}
\newcommand{\HHopt}{Harary--Hill optimal}
\newcommand{\Triangle}{3-cycle}

\title{Bishellable drawings of $K_n$}

\author{Bernardo M.~\'Abrego\thanks{Department of Mathematics,
California State University, Northridge, {\tt bernardo.abrego@csun.edu}}
\and
Oswin Aichholzer\thanks{Institute of Software Technology,
Graz University of Technology, {\tt oaich@ist.tugraz.at}}
\and
Silvia Fern\'andez-Merchant\thanks{Department of Mathematics,
California State University, Northridge, {\tt silvia.fernandez@csun.edu}}
\and
Dan McQuillan\thanks{Department of Mathematics, Norwich University, {\tt dmcquill@norwich.edu}}
\and
Bojan Mohar\thanks{Department of Mathematics, Simon Fraser University,
  {\tt mohar@sfu.ca}
  \newline\indent\indent On leave from: IMFM, Ljubljana, Slovenia.}
\and
Petra Mutzel\thanks{Department of Computer Science, Technische
  Universit\"at Dortmund, {\tt petra.mutzel@tu-dortmund.de}}
\and
Pedro Ramos\thanks{Departamento de F\'isica y Matem\'aticas,
Universidad de Alcal\'a, {\tt pedro.ramos@uah.es}}
\and
R. Bruce Richter\thanks{Department of Combinatorics and Optimization,
  University of Waterloo, {\tt brichter@uwaterloo.ca}}
\and
Birgit Vogtenhuber\thanks{Institute of Software Technology,
Graz University of Technology, {\tt bvogt@ist.tugraz.at}}}

\index{\'Abrego, Bernardo}
\index{Aichholzer, Oswin}
\index{Fern\'andez-Merchant, Silvia}
\index{McQuillan, Dan}
\index{Mohar, Bojan}
\index{Mutzel, Petra}
\index{Ramos, Pedro}
\index{Richter, R. Bruce}
\index{Vogtenhuber, Birgit}

\begin{document}

\maketitle

\begin{abstract}
\noindent
The Harary--Hill conjecture, still open after more than 50 years, asserts that the crossing number of the complete graph $K_n$ is
\vspace{-1ex}
\[
H(n) = \frac 1 4 \left\lfloor\frac{\mathstrut n}{\mathstrut 2}\right\rfloor
\left\lfloor\frac{\mathstrut n-1}{\mathstrut 2}\right\rfloor
\left\lfloor\frac{\mathstrut n-2}{\mathstrut 2}\right\rfloor
\left\lfloor\frac{\mathstrut n-3}{\mathstrut 2}\right\rfloor\,.
\]
\'Abrego et al. [B.~M. {\'{A}}brego, O. Aichholzer, S. Fern{\'{a}}ndez{-}Merchant,
  P. Ramos, and G. Salazar. Shellable drawings and the cylindrical crossing number of
  $K_n$. {\em Disc. {\&} Comput. Geom.}, 52(4):743--753, 2014.]
introduced the notion of shellability of a drawing $D$ of $K_n$. They proved that if $D$ is $s$-shellable for some $s\geq\lfloor\frac{n}{2}\rfloor$, then $D$ has at least $H(n)$ crossings.  This is the first combinatorial condition on a drawing that guarantees at least $H(n)$ crossings.

In this work, we generalize the concept of $s$-shellability to bishellability,
where the former implies the latter in the sense that every $s$-shellable drawing is,  for any $b \leq s-2$, also \mbox{$b$-bishellable}.
Our main result is that $(\lfloor \frac{n}{2} \rfloor\!-\!2)$-bishellability of a drawing $D$ of $K_n$ also guarantees, with a simpler proof than for \mbox{$s$-shellability},
that $D$ has at least $H(n)$ crossings.
We exhibit a drawing of $K_{11}$ that has $H(11)$ crossings, is 3-bishellable, and is not $s$-shellable for any $s\geq5$.
This shows that we have properly extended the class of drawings for which the Harary--Hill Conjecture is proved.
Moreover, we provide an infinite family of drawings of $K_n$ that are $(\lfloor \frac{n}{2} \rfloor\!-\!2)$-bishellable, but not $s$-shellable for any $s\geq\lfloor\frac{n}{2}\rfloor$.  
\end{abstract}

\section{Introduction}\label{sintro}

We consider drawings of the complete graph $K_n$ in the plane in which vertices are drawn as points in the plane and edges as simple planar curves that contain no vertices other than their endpoints.
As usual, we require that all intersections are proper crossings (no tangencies)
and that two edges share only a finite
number of points.
The number  $\crn(D)$ of crossings in a drawing $D$ is the sum of the number of intersection points of all unordered pairs of interiors of edges.
The \emph{crossing number} $\crn(G)$ is the minimum $\crn(D)$ over all drawings $D$ of $G$. A drawing is \emph{crossing optimal (or minimal)} if $\crn(D)=\crn(G)$.

A long-standing conjecture is that the crossing number $\crn(K_n)$ of the complete graph $K_n$ is equal to
$$H(n):=\frac 1 4 \left\lfloor\frac{\mathstrut n}{\mathstrut 2}\right\rfloor
\left\lfloor\frac{\mathstrut n-1}{\mathstrut 2}\right\rfloor
\left\lfloor\frac{\mathstrut n-2}{\mathstrut 2}\right\rfloor
\left\lfloor\frac{\mathstrut n-3}{\mathstrut 2}\right\rfloor\,.$$
A very fine history of this and related problems is given by Beineke and Wilson~\cite{bw}.
They attribute the conjecture to Anthony Hill. As it is first published by Harary and Hill in~\cite{hh}, we propose the notation $H(n)$ used above to denote the conjectured value of $\crn(K_n)$ and attribute the conjecture to Harary--Hill.

\noindent
According to~\cite{bw}, Hill proposed a construction with vertices on two concentric cycles and conjectured $H(n)$ to be the number of crossings for those drawings. Bla\v{z}ek and Koman~\cite{bk} proved the following variation of Hill's construction. Half of the vertices are drawn evenly spaced (as a cycle) at the rim of the top lid of a cylinder (tin can) and the remaining vertices are drawn evenly spaced at the rim of the bottom lid.
The edges are drawn either as straight-lines within the cylinder lids or (in the case they connect two vertices from different lids) as shortest geodesic lines on the cylinder. This model is identical to the model with two concentric circles and it gave rise to a name for a whole class of drawings, see below.

So far, the conjecture has been verified for $n\le 12$ only. Guy established the result for $n\leq 10$~\cite{guy1972crossing} and Pan and Richter~\cite{PanR07} have shown that $\crn(K_{11})$ is 100 using a computer proof; well-known counting arguments show $\crn(K_{12})=150$.
McQuillan and Richter~\cite{McQR14} present a proof that $\crn(K_9)=36$ not based on exhaustive case analysis.

Recently, an important line of research has been started by \'Abrego et al.~\cite{twopage}, who restricted the allowed drawing styles and proved that for these drawings the conjecture is true. In \cite{twopage}, they consider 2-page book drawings of $K_n$. We recall that in a 2-page book drawing of a graph all vertices are on a line $\ell$ and each edge needs to be drawn on one of the two half-planes defined by $\ell$. In~\cite{monotone,balko}, the technique was extended to drawings of $K_n$ where all the vertices have different $x$-coordinates and the edges are $x$-monotone curves, known as \emph{monotone\/} drawings.
In~\cite{shell1}, \'Abrego et al.\ generalize their result to cylindrical drawings and $x$-bounded drawings.
A \emph{cylindrical drawing\/} has two concentric circles on which all the vertices must be placed, and no edge intersects these circles. An \emph{$x$-bounded drawing} requires that all edges are contained within a strip bounded by the vertical lines defined by their endpoints.
Every monotone drawing is $x$-bounded and every 2-page book drawing is a cylindrical drawing, so their result indeed generalizes previous results.
Their result, which also prompted this work, was based on the first general combinatorial condition on a drawing $D$ of $K_n$ that guarantees that $D$ has at least $H(n)$ crossings.
For this, they introduced the notion of \emph{shellability} of a drawing of $K_n$.

In a related work, Balko et al.~\cite{balko} give a combinatorial characterization of several classes of $x$-monotone drawings of complete graphs and show that also the odd crossing number (a different variant of counting crossings) of $x$-monotone drawings as well as shellable drawings of $K_n$ is at least $H(n)$.

The purpose of this work is to define a more general version of shellability that we call bishellability, and which is implied by shellability.
The main benefit of our approach is the simplification of the principal concept.
This allows a somewhat simpler and more intuitive proof for the fact that bishellable drawings satisfy the Harary--Hill Conjecture.
Moreover, bishellability reflects better the required properties. 
We are convinced that this is a further step to gain more insight into the structure of crossing minimal drawings, with the ultimate goal to prove the Harary--Hill Conjecture.

For the following definition, we recall that, for a drawing $D$ of $K_n$, a {\em face\/} of~$D$ is a component of $\mathbb R^2\setminus D$.  This is the same notion as for embeddings:  if we convert each crossing point of $D$ into a vertex, then the faces of $D$ are the faces of the planarly embedded graph. Finally, if $V$ is a subset of vertices in the drawing, $D-V$ denotes the drawing obtained when vertices of $V$ and all edges incident to them are deleted from $D$.

\begin{definition}[\cite{shell1}]
	For a positive integer $s$, a planar drawing $D$ of $K_n$ is \mbox{\emph{$s$-shellable}} if there is a sequence $v_1,v_2,\dots,v_s$ of distinct vertices of $D$ so that, relative to a reference face $F$, for all integers $r,t$ with $1\le r<t\le s$, the vertices~$v_r$ and $v_t$ are both incident with the face of $D-\{v_1,\dots,v_{r-1},v_{t+1},\dots,v_s\}$ containing $F$.
\end{definition}

The main theorem in \cite{shell1} is the following.

\begin{theorem}[\cite{shell1}]\label{th:old}
Let $D$ be a drawing of $K_n$.  If there is an integer $s\ge \lfloor \frac{n}{2}\rfloor$ such that $D$ is $s$-shellable, then $cr(D)\ge H(n)$.
\end{theorem}

One of the disadvantages of the notion of shellability is that $s$-shellable does not imply $(s\!-\!1)$-shellable.  This is because the sequence $v_1,v_2,\dots,v_s$ of vertices needs to be circular to get from the reference face $F$ and back to $F$ again, and a long circular sequence does not imply a shorter circular sequence.

We introduce a more general variant of shellability that we call \emph{bishellability}.
\begin{definition}
	For a non-negative integer $s$, a drawing $D$ of $K_n$ is \emph{$s$-bishellable} if there exist sequences $a_0,a_1,\dots,a_s$ and $b_s,b_{s-1},\dots,b_1,b_0$, each sequence consisting of distinct vertices of $K_n$, so that, with respect to a reference face $F$:
\begin{enumerate}[label=(\arabic*)]
	\item for each $i=0,1,2,\dots,s$, the vertex $a_i$ is incident with the face of $D-\{a_0,a_1,\dots,a_{i-1}\}$ that contains~$F$;
	\item for each $i=0,1,2,\dots,s$, the vertex $b_i$ is incident with the face of $D-\{b_0,b_1,\dots,b_{i-1}\}$ that contains~$F$; and
	\item for each $i=0,1,\dots,s$, the set $\{a_0,a_1,\dots a_i\}\cap \{b_{s-i},b_{s-i-1},\dots,b_0\} = \emptyset$.
\end{enumerate}
\end{definition}

We remark that if $a_0,a_1,\dots,a_s$ and $b_s,b_{s-1},\dots,b_0$ show that $D$ is $s$-bishellable, then the same sequences without $a_s$ and $b_s$ show that $D$ is $(s\!-\!1)$-bishellable. Moreover, the vertices $a_0$ and $b_0$ must lie on the boundary of the common face $F$

Also, if $D$ is $s$-shellable, with witnessing sequence $v_1,v_2,\dots,v_s$, then $D$ is $(s\!-\!2)$-bishellable with witnessing sequences $a_0,a_1,\dots,a_{s-2}$ and $b_{s-2},b_{s-3},\dots,b_0$ defined by $a_i=v_{i+1}$ and $b_i=v_{s-i}$.

Here is the version of Theorem~\ref{th:old} that holds for bishellable drawings; its proof is in the next section.

\begin{theorem}\label{th:main}
If $D$ is an $(\lfloor \frac{n}{2}\rfloor-2)$-bishellable drawing of $K_n$, then $\crn(D)\ge H(n)$.
\end{theorem}

There are two main remarks to be made here.  First is the pleasant feature that the theorem requires only one value for the amount of bishellability of the drawing.  Second, even though the same principal ideas are used in the two proofs, the proof of Theorem~\ref{th:main} involves a simpler induction than the proof of Theorem~\ref{th:old} in~\cite{shell1}. Both facts are due to the monotonicity of the new definition of bishellability.

To simplify the discussion, we define a drawing $D$ of $K_n$ to be
\emph{shellable} if it is $s$-shellable for some $s \geq \lfloor \frac{n}{2}
\rfloor$, and \emph{bishellable} if it is $(\lfloor \frac{n}{2}
\rfloor\!-\!2)$-bishellable. That is, shellable and bishellable drawings
have at least $H(n)$ crossings. Furthermore, we call a drawing of
$K_n$ \emph{\HHopt\ } if it has $H(n)$ crossings.
We use this notation to keep in mind that drawings with $H(n)$ crossings are only conjectured to be optimal.

Besides the proof of Theorem~\ref{th:main}, we show that bishellability leads to an extended class of drawings for which the Harary--Hill Conjecture is true. 
The drawing $D$ of $K_{11}$ in \figurename~\ref{fig:n11bi} is bishellable (so Theorem~\ref{th:main} shows $\crn(D)\ge H(11) $), but not shellable (so Theorem~\ref{th:old} does not apply). 
Moreover, we show the following theorem; 
see Section~\ref{sec:drawings} for the proof.
\begin{theorem}\label{thm:extension}
	There exists an infinite family of drawings of complete graphs that are bishellable but not shellable.
\end{theorem}

\begin{figure}[htb]
	\centering
	{\includegraphics[page=3,scale=0.92]{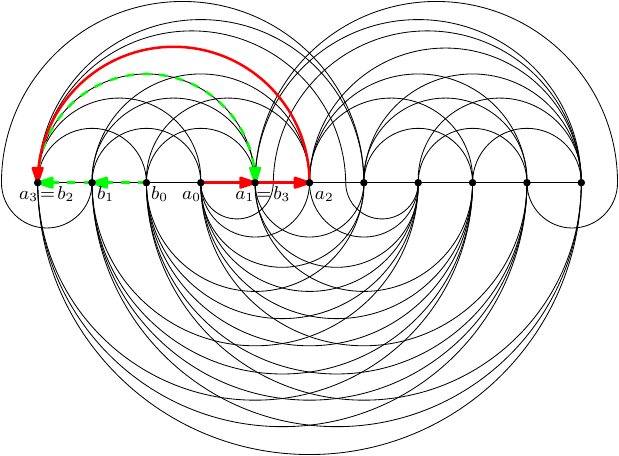}}
\caption{\label{fig:n11bi}
	A \HHopt\ drawing of $K_{11}$ that is bishellable but not $s$-shellable
        for any $s \geq 5$.}
\end{figure}

\section{Bishellable Drawings and the Crossing Number}
In this section, we recall the notion of simple drawings and $k$-edges,
and how to express the crossing number of a drawing $D$ in terms of 
a weighted sum of the numbers of $k$-edges, for $k=0,1,2,\dots,\lfloor \frac{n}{2}\rfloor-1$.
Then we prove that, under the assumption that a drawing $D$ is bishellable, for each relevant $k$, the relevant weighted sum of $k$-edges is large enough to prove that there are at least $H(n)$ crossings in $D$.

\subsection{Preliminaries}

We remark that we consider only {\em simple\/} drawings (also known as good drawings, or simple topological drawings): no edge can cross itself; no two edges with a common incident vertex can cross; and no two edges can cross each other more than once.  It is well-known that any drawing of a graph $G$ with fewest crossings is necessarily simple.

The relation between the number of crossings in a rectilinear (or pseudo-linear) drawing of $ K_n $ and the number of its $ k $-edges was first described by Lov\'asz et al.~\cite{lve} and, independently, by \'Abrego and Fern\'andez~\cite{af}.   \'Abrego et al.~\cite{twopage} generalized the notion of $k$-edges to arbitrary simple drawings, as follows. Fix a drawing $ D $ of $ K_n $ and a face $ F $ of $ D $.    For an edge $uv$ of $D$, we arbitrarily choose one of its orientations:  here we use $u$ to $v$.  For each other vertex $w$, the three vertices $u,v,w$ induce a \Triangle\ $T$ in $ D $.  Because $D$ is a simple drawing, $T$ is a simple closed curve in the sphere (it is slightly simpler technically to consider embeddings in the sphere rather than the plane).

As we traverse $T$ from $u$ to $v$ to $w$ and back to $u$, $T$ has natural right and left sides relative to the directed edge $uv$.  Assign to $w$ the side $R$ (right) if $F$ is on the left side of $T$; otherwise, assign $L$ (left) to $w$. Thus, for each $w\notin \{u,v\}$, $w$ is assigned either $R$ or $L$.  Then $uv$ is a \emph{$k$-edge\/} if $k$ is the smaller of the number of $R$'s and $L$'s (for $uv$). Note that being a $ k $-edge is independent of the orientation of the edge $ uv $ as reversing $ uv $ simply exchanges all the labels $ R $ and $ L $.

We make one small observation that helps the later discussion:  every edge~$e$ that has a segment incident with $F$ is a 0-edge.  To see this, suppose some segment $\tilde e$ of $e$ is incident with $F$ and we orient $e$ so that $F$ is to the left of~$\tilde e$.  Then, for every $w$ not incident with $e$, the \Triangle\ determined by $e$ and $w$ will have $F$ on the left side, showing $w$ is an $R$.  That is, all vertices are $R$ with respect to $e$, so $e$ is a 0-edge, as claimed.

\subsection{Proof of Theorem~\ref{th:main}}\label{sec:mainproof}
One main ingredient of the proof is the relation between
the number of $k$-edges and the number of crossings in a drawing.
This relation was shown for rectilinear drawings in \cite{af,lve} and extended to simple drawings in \cite{twopage},
where it is shown that the number of crossings of a simple drawing can be expressed as a weighted sum of the number of $k$-edges of the drawing.
We include this discussion in the appendix, for the sake of completeness.

Specifically, if we denote by $E_k(D)$ the number of $k$-edges of $D$ and consider
\[
E_{{\leq}k}\left(  D\right)  :=\sum\limits_{j=0}^{k}E_{j}\left(
D\right)
\]
and
\begin{equation}\label{eq:D-sum}
E_{{\leq}{\leq}k}(D)  :=\sum\limits_{j=0}^{k}E_{\leq j}\left(
D\right)  =\sum\limits_{j=0}^{k}\sum\limits_{i=0}^{j}E_{i}\left(
D\right) =\sum\limits_{i=0}^{k}\left(  k+1-i\right)  E_{i}\left(
D\right)\,.
\end{equation}
In \cite{twopage} it is shown that
\begin{equation}\label{eq:twopage}
\crn(D)=2\sum\limits_{k=0}^{\lfloor
	\frac{n}{2}\rfloor-2}E_{{\leq}{\leq}k}(D)-\frac
{1}{2}\binom{n}{2}\left\lfloor \frac{n-2}{2}\right\rfloor -\frac{1}%
	{2}(1+(-1)^{n})E_{\leq\leq\left\lfloor \frac{n}{2}\right\rfloor -2}(D)\,.
\end{equation}
Therefore, a straightforward calculation translates any lower bound on $E_{{\leq}{\leq}k}(D)$ into a lower bound for $\crn(D)$ and, specifically, showing that $\displaystyle E_{{\leq}{\leq}k}(D) \geq 3\binom{k+3}{3}$ for $0 \leq k \leq \lfloor\frac{n}{2}\rfloor- 2$ implies $\crn(D)\geq H(n)$.

\begin{lemma}\label{lm:inequality}
If a drawing $D$  of $K_n$ is $k$-bishellable and $0 \le k\le \lfloor \frac{n}{2}\rfloor- 2$, then
\[
E_{\leq \leq k}(D)\ge 3 {k+3\choose 3}.
\]
\end{lemma}

\begin{proof}[Proof of Theorem~\ref{th:main}]
Since $D$ is $(\lfloor \frac{n}{2}\rfloor\!-\!2)$-bishellable, for each $k=0,1,\dots, \lfloor \frac{n}{2}\rfloor -2$, the drawing $D$ is $k$-bishellable.  Therefore, 
Lemma~\ref{lm:inequality} implies that,
\[
 \text{ for all} \ k\in \{0,1,\dots,\left\lfloor \frac{n}{2}\right\rfloor-2\}, \quad E_{\leq \leq k}(D)\ge 3{k+3\choose 3}\,,
\]
as required.  
Plugging these lower bounds into Equation~\ref{eq:twopage} 
the desired lower bound of $H(n)$ on the number of crossings in $D$ follows (for details see the proof of Theorem~3 in~\cite{twopage}).
\end{proof}

\begin{proof}[Proof of Lemma~\ref{lm:inequality}]
We essentially follow the ideas of the proofs in~\cite{shell1}, but use a simpler and more direct approach provided by the new concept of bishellability.
We proceed by induction on $k$. 
The base case of $k=0$ is trivial, as the reference face $F$ is incident with at least three edges and each of these is a 0-edge.  
Thus,
\[ E_{\leq \leq 0}(D) = \sum_{i=0}^0(0+1-i)E_i(D) = E_0(D)\ge 3=3{0+3\choose 3}\,, \]
as required.

For the induction step, let $a_0,a_1,\dots,a_k$ and $b_k,b_{k-1},\dots,b_0$ be sequences witnessing $k$-bishellability and consider the drawing $D-a_0$.  Then $a_1,\dots,a_k$, $b_{k-1},\dots,b_0$ show it is $(k\!-\!1)$-bishellable and, since $k-1\le (\lfloor \frac{n}{2}\rfloor -2)-1 \le \lfloor \frac{n-1}{2}\rfloor -2$, the induction implies that

\[ E_{\leq \leq k-1}(D - a_0) = \sum_{i=0}^{k-1} ((k-1)+1-i)E_i(D-a_0)\ge 3 {(k-1)+3\choose 3}\,, \]
 which can be rewritten as 
\begin{equation}\label{eq:D-a-sum}
 	E_{\leq \leq k-1}(D - a_0) = \sum_{i=0}^{k-1} (k-i)E_i(D-a_0)\ge 3 {k+2\choose 3}\,.
\end{equation}

Consider an edge $e$ in $D-a_0$.  If $e$ is an $i$-edge in $D-a_0$ with $i\le \lfloor \frac{n-1}{2}\rfloor -2$, then it is either an $i$-edge or an $(i+1)$-edge of $D$, depending on whether $a_0$ joins the majority or minority part of the $R$'s and $L$'s with respect to $e$ in $D-a_0$.  We call those that are $i$-edges in both $D-a_0$ and $D$ \emph{invariant}.

Note that the coefficient of a non-invariant $i$-edge in the sum for $E_{\leq \leq k-1}(D - a_0)$ in Equation~(\ref{eq:D-a-sum}) is $k-i$. Its coefficient in the sum for $E_{\leq \leq k-1}(D)$ in Equation~(\ref{eq:D-sum}) is $(k+1)-(i+1) = k-i$.  Thus, its contribution to both sums is the same.  On the other hand, an invariant $i$-edge contributes $k-i$ to the sum  for $E_{\leq \leq k-1}(D - a_0)$ in Equation~(\ref{eq:D-a-sum}) and $k+1-i$ to the sum  for $E_{\leq \leq k-1}(D)$ in Equation~(\ref{eq:D-sum}).

There is an additional contribution to the $D$-sum from the edges incident with~$a_0$, which we will discuss shortly.
Hence, we altogether obtain
\[
\sum_{i=0}^k (k+1-i)E_i(D)\ge \sum_{i=0}^{k-1}(k-i)E_i(D-a_0) + |\textrm{invariant edges}| + \textrm{contribution of $a_0$-edges}\,.
\]
We shall prove that there are at least $k+2\choose 2$ invariant edges and the contribution of the edges incident with $a_0$ is at least $2{k+2\choose 2}$.  Together with the induction assumption applied to $\sum_{i=0}^{k-1}(k-i)E_i(D)$, we conclude that
\[ E_{\leq \leq k} (D) = \sum_{i=0}^k(k+1-i)E_i(D) \ge 3{k+2\choose 3} + 3{k+2\choose 2} = 3{k+3\choose 3}\,, \]
as required.

Next we consider the edges incident with $a_0$. Let $e_0$ and $e'_0$ be the two edges incident with $a_0$ 
that are incident with the reference face $F$ in some small environment of $a_0$. 
Consequently, the fact that $k\le (\frac{n}{2})-2$ shows that we may write the cyclic rotation of the edges incident with $a_0$ as $(e_0,e_1,\dots,e_k,\dots,e'_k,e'_{k-1},\dots,e'_1,e'_0)$.

To make the discussion uniform, orient all edges incident with $a_0$ away from~$a_0$.  Consi\-deration of any \Triangle\ $(a_0,u,v)$ shows that, if, relative to the edge $a_0u$, $v$ is $L$, then, relative to $a_0v$,  $u$ is~$R$.

We arbitrarily choose the orientation of the sphere so that all the vertices not incident with $e_0$ are $R$'s for $e_0$.
Because deleting $e_0,e_1,\dots,e_{i-1}$ puts $e_i$ into the boundary of the (extended) reference face, we see that $e_i$ has at most $i$ $L$'s and all the rest are $R$'s.  From this it is immediate that $e_i$ is an $\leq i$-edge, as required.

In particular, letting $j$ be the integer such that $e_i$ is a $j$-edge, we know that $e_i$ contributes  $k+1-j$ to the sum $\sum_{i=0}^k(k+1-i)E_i(D)$.  Since $j\le i$, $k+1-j\ge k+1-i$, $e_i$ contributes at least $k+1-i$ to the sum.  Therefore, $e_0,e_1,\dots,e_k$ contribute at least $(k+1)+k+\cdots+2+1={k+2\choose 2}$ to the sum $\sum_{i=0}^k(k+1-i)E_i(D)$.  Likewise, $e'_0,e'_1,\dots,e'_k$ contribute at least the same amount, as desired for the contribution from the edges incident with $a_0$.

The invariant edges are determined by the $b_i$'s.  The main point is that $b_0$ is incident with at least $k+1$ invariant edges upon deletion of $a_0$.
To see this, we observe that if $f_0$ and $f'_0$ are the edges incident with $b_0$ 
that are incident with $F$ in some small environment of $b_0$ then
the symmetry $a_0\leftrightarrow b_0$ in the definition of bishellable drawings implies that the cyclic rotation of edges at $b_0$ can be taken as $(f_0,f_1,\dots,f_k,\dots,f'_k,f'_{k-1},\dots,f'_1,f'_0)$.

We may choose the labelling of the $f_i$s and $f'_i$s so that $a_0b_0$ is not one of $f_0,f_1,\dots,f_k$.
Then, by the same arguments as for $e_i$ before, we know that for every $0\le i\le k$, $f_i$ is an $\leq i$-edge,
as only the non-$b_0$ endpoints of $f_0,f_1,\dots,f_{i-1}$ are possible $L$'s (say) while all the remaining vertices must be $R$'s.
In particular, $a_0$ is an $R$ for all of $f_0,f_1,\dots,f_k$.  It follows that, for some $j\le i$, $f_i$ is a $j$-edge in both $D$ and $D-a_0$.  That is, $b_0$ is incident with at least $k+1$ invariant edges.
	
Here is the other significant simplification due to bishellability.  The identical argument applies to the vertex $b_i$ in the $(k-i)$-bishellable drawing $D_i= \linebreak D-\{b_0,b_1,\dots,b_{i-1}\}$, as witnessed by the sequences $a_0, a_1, \dots, a_{k-i}$ and \linebreak $b_k, b_{k-1}, \dots, b_i$.  The vertex $b_i$ (in the role corresponding to $b_0$ in the preceding paragraph) is incident with at least $(k-i+1)$ edges invariant (relative to $D_i$ and $D_i-a_0$).   The values ($R$ or $L$) of $b_{i-1},b_{i-2},\dots,b_0$ are, for each of the edges incident with~$b_i$, independent of whether we consider $D_i$ or $D_i-a_0$.  Thus, each edge of $D_i$ incident with $b_i$ is a $j$-edge of $D$ if and only if it is a $j$-edge in $D-a_0$.  It follows that $D$ has, in total, at least $(k+1)+k+\cdots+1 = \binom{k+2}2$ invariant edges incident with $b_k,b_{k-1},\dots,b_0$, as required.
\end{proof}

\section{Bishellable and Non-bishellable Drawings}
\label{sec:drawings}

So far, our goal was to to  
simplify the notion of shellability and to simplify the proof that shellable drawings of $K_n$ have at least $H(n)$ crossings.  
However, there is also 
interest in understanding the distinctions between shellable, bishellable, and general drawings.

Two Harary--Hill optimal drawings of $K_n$ with $n\ge 11$ odd are given in~\cite{nonshell}. 
One especially relevant to us has every edge crossed at least once. 
\figurename~\ref{fig:n11non_n09non}~(left) depicts this drawing for $K_{11}$. 
In particular, no face of this drawing is incident with two vertices, implying that the drawing cannot be $s$-bishellable for any $s\ge 0$.  
The smallest complete graph where a non-bishellable drawing exists is $K_6$. 
However, that example is not \HHopt. 
The smallest \HHopt\ drawing which is not bishellable is 
one of $K_9$; see \figurename~\ref{fig:n11non_n09non}~(right). 
Both latter drawings have been obtained by exhaustive computations.
Moreover, by~\cite{nonshell}, there are Harary--Hill optimal drawings with arbitrarily many vertices that are neither shellable nor bishellable.

\begin{figure}[htb]
	\centering
	{\includegraphics[page=1,scale=0.9]{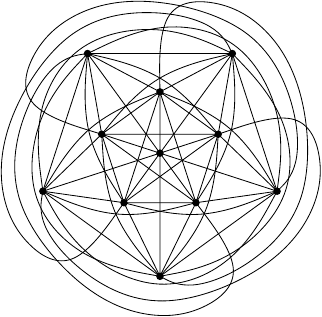}} \hspace{1cm}
	{\includegraphics[page=3,scale=0.9]{n09_nonshellable.pdf}}
\caption{\label{fig:n11non_n09non}
	A non-bishellable \HHopt\  drawing of $K_{11}$ from~\cite{nonshell} where all edges are
		crossed (left) and  
	a \HHopt\ drawing of $K_9$ that is not bishellable (right).}
\end{figure}

On the other hand, any cylindrical drawing of $K_n$ has a cycle of length at least $\lceil \frac{n}{2}\rceil$ having no edges crossed.  
Such a cycle shows that the drawing has a shelling sequence of length at least $\lceil \frac{n}{2}\rceil$. 
Hence it is shellable and also bishellable.  

To distinguish between shellable and bishellable drawings is more subtle.  
The smallest example of a \HHopt\ drawing that is not shellable but bishellable is the drawing of $K_{11}$ shown in \figurename~\ref{fig:n11bi}. 
The two sequences $a_0,a_1,a_2,a_3$ and $b_3,b_2,b_1,b_0$ proving bishellability are indicated in the drawing. 
The example is symmetric and almost a 2-page book drawing or a monotone drawing. 
There are just two edges of the spine that are crossed and two edges that are non-monotone.
It is not straightforward to see that this drawing is indeed not shellable. 
Thus, to confirm non-shellability we used computations, checking all possible shelling sequences.

\noindent
Intuitively, the main difference between shellability and bishellability is that shellability requires a closed cycle of vertices as shelling sequence, while bishellability allows ``flying'' ends of the two sequences. 
We now make use of this fact to show Theorem~\ref{thm:extension}, which states the existence of an infinite number of drawings of complete graphs that are bishellable but not shellable.

\begin{proof}[Proof of Theorem~\ref{thm:extension}]
Consider first the gadget depicted in \figurename~\ref{fig:gadget01}. The drawn edges form two faces that are bounded by 3 edges each (two edges incident to $v$ and a third one that intersects the first two), each acting as a ``wall'' for shellability.
To see that, assume that the starting face for a shelling sequence lies in the exterior, i.e., in the area~$A$. 
In order to reach vertices from area $B$ during the shelling process, at least one of the edges bounding the inner face (green dashed edges) must be removed.  

\begin{figure}[htb]
	\centering
	{\includegraphics[page=1,scale=1]{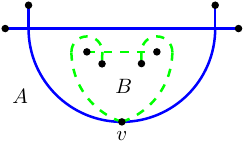}}
\caption{\label{fig:gadget01}
	Gadget with two walls (drawn dashed green and solid blue) and one gate $v$ that blocks any shelling cycle containing points from both $A$ and $B$. 
	}
\end{figure}

Observe that $v$ is the only vertex which is accessible from the exterior. 
Hence, $v$ is the  unique ``gate'' for any shelling sequence to continue from vertices of $A \setminus \{v\}$ to vertices of $B \setminus \{v\}$. 
But shellability requires that we can shell along a cycle in both directions, i.e., we would need two gates like~$v$. 
Thus, a shelling sequence that starts in $A$ (including vertex $v$) and contains at least one vertex of $A \setminus \{v\}$ cannot contain any vertices from~$B \setminus \{v\}$.
Similarly, a shelling that starts in $B$ (including $v$) is blocked by the wall formed by the boundary of the outer face (blue solid edges) of the gadget, where again vertex $v$ builds the unique gate. 
Thus, we conclude that any shelling sequence can contain either vertices from $A \setminus \{v\}$  or $B \setminus \{v\}$, but not both ($v$ can be contained in both cases).
Note, however, that this restriction does not apply to bishellability, as we do not require the sequence to work in both directions and thus one gate is sufficient. 

Our goal is to extend the gadget with vertices in the regions $A$ and $B$ such that any 
possible shelling sequence of length at least $\lfloor \frac{n}{2}\rfloor$ 
must contain vertices in each of the regions $A$ and $B$. 
In particular, then $v$ must be an internal vertex of the sequence and hence $A$ and $B$ must each contain one end of the sequence. 
Since for shellability, the ends of the sequence must be incident with the same face of the drawing, this then constitutes a proof of non-shellability. 

\begin{figure}[htb]
	\centering
	{\includegraphics[page=1,scale=1]{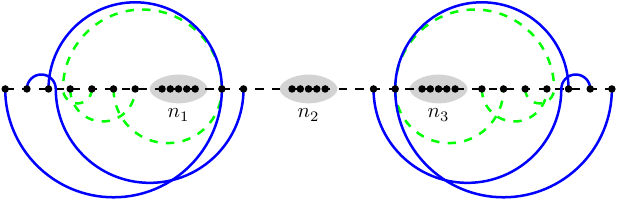}}
\caption{\label{fig:family_nonshellbishell}
	Base construction for non-shellable 
	but bishellable drawings. 
	}
\end{figure}

Consider the drawing in \figurename~\ref{fig:family_nonshellbishell}. It contains two copies of the gadget from \figurename~\ref{fig:gadget01} and is drawn in a way that all vertices are placed on a horizontal line called ``spine'', similar to monotone or 2-page book drawings.
In addition to the two gadgets, $n_2$ points are placed in the region corresponding to $A$ in \figurename~\ref{fig:gadget01} and $n_1$ points and $n_3$ points, respectively, are placed inside the regions corresponding to $B$. 
The above arguments imply that if $n_1+5 < \lfloor \frac{n}{2} \rfloor$, $n_2+10 < \lfloor \frac{n}{2} \rfloor$, and $n_3+5 < \lfloor \frac{n}{2} \rfloor$ then the drawing is not $s$-shellable for any $s \geq \lfloor \frac{n}{2} \rfloor$. Concrete values for general $n$ are for example $n_1 = \lfloor \frac{n-11}{3} \rfloor$, $n_2 = \lfloor \frac{n-27}{3} \rfloor$, and $n_3 = \lfloor \frac{n-13}{3} \rfloor$ with $n \geq 27$. But also other cardinalities and smaller sets are possible, for example $n_1=n_3=2$, $n_2=0$ and thus $n=22$. 
On the other hand, the drawing is $(\lfloor \frac{n}{2} \rfloor-1)$-bishellable by the following sequences: 
start with two neighboring vertices in the middle of the central set (the one containing the $n_2$ points) and continue with the vertices to the left and right, respectively, according to their order along the spine. 

\noindent
Further, 
the drawing can be completed to a simple drawing of $K_n$ in many ways such that 
the resulting drawing is still bishellable with the just described sequences.
In the following, we give a description of an example set of valid extensions, 
where most edges are drawn as half-circles, similar to the usual way of 2-page book drawings, and some edges are drawn as $s$-shaped arcs consisting of two half circles.
Note that in  \figurename~\ref{fig:family_nonshellbishell}, exactly two edges of each gadget from \figurename~\ref{fig:gadget01} are drawn as a combination of two half-circles, both incident to the gate-vertex, while the rest of the edges is drawn as half-circles.
In \figurename~\ref{fig:family_forcededges}, the endpoints of all edges that are not half-circles are marked with red (larger) disks and labeled. 
We draw the edge $p_1p_2$ as $s$-shaped arc such that it doesn't cross with $p_vp_1$ and $p_vp_2$.
Likewise, we draw the edge $q_1q_2$ as $s$-shaped arc such that it doesn't cross with $q_vq_1$ and $q_vq_2$.
The remaining edges are drawn as half circles, where for some of them, the side of the spine is fixed by the following rules:
(i) some remaining edges incident with $p_1$ or $q_1$ have to be drawn above the spine to avoid crossings with $p_1p_3$ or $q_1p_3$,
(ii) some remaining  edges incident with $p_2$ or $q_2$ have to be drawn below the spine to avoid crossings with $p_2p_v$ or $q_2q_v$,
(iii) one remaining  edge each incident with $p_v$ and $q_v$, respectively, has to be drawn above the spine to avoid crossings with $p_2p_v$ or $q_2q_v$, and
(iv) some remaining  edges incident with $p_3$ and $q_3$ have to be drawn below the spine to avoid crossings with $p_1p_3$ or $q_1q_3$.
\figurename~\ref{fig:family_forcededges} depicts the edges that are forced to be $s$-shaped or on one fixed side. 
All remaining edges can be drawn as half circles on either side of the spine.
\end{proof}

\begin{figure}[htb]
	\centering
	{\includegraphics[page=4,scale=1]{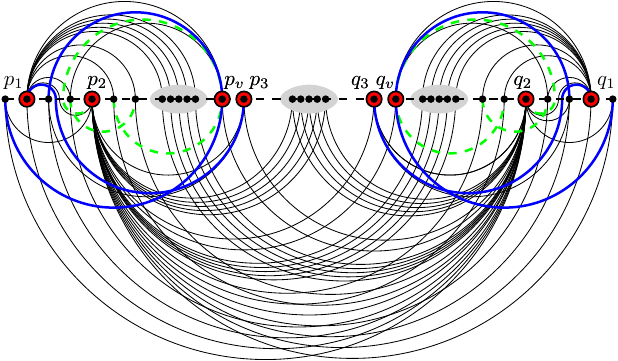}}
\caption{\label{fig:family_forcededges}
	Forced edges when extending the drawing of \figurename~\ref{fig:family_nonshellbishell} to a bishellable drawing of $K_n$. 
	}
	\vspace{-1ex}
\end{figure}

Finally, note that 
in any drawing according to the construction in the above proof, only two of the edges, namely, $p_2p_v$ and $q_2q_v$, are non-monotone. 
So the constructed drawings are 
very close to drawings that are known to be shellable.

\section{Conclusion}

Recent progress on the rectilinear crossing number of $K_n$~\cite{aichholzer2007new,abrego2012k} has depended on making more refined estimates than that provided by Lemma~\ref{lm:inequality}.  These refinements cannot occur in the context of the (topological) crossing number considered in this paper, as the bound in Lemma~\ref{lm:inequality} cannot be improved for a drawing having $H(n)$ crossings.

Our computations show that about $43.8\%$ of the 403,079 (up to weak isomorphism) \HHopt\ drawings of $K_{11}$ are not bishellable.
It seems likely that, as $n$ grows, the proportion of bishellable \HHopt\ drawings vanishes.
However, Theorem~\ref{th:main} (indeed Theorem~\ref{th:old} is enough) implies, for example, that any drawing with $K_{\frac{n}{2}}$ drawn in the southern hemisphere (using an uncrossed $\frac{n}{2}$-cycle on the equator) and the remaining $K_{\frac{n}{2}}$ drawn in the northern hemisphere (arbitrarily) must have at least  $H(n)$ crossings.
These kinds of drawings play a role in the computer-free proof~\cite{McQR14} that $\crn(K_9)=36$.

We close with some open problems:
\begin{itemize}
\item Can we construct a family of \HHopt\ drawings of $K_n$
  which are non-shellable, but bishellable, similar to the
  constructions in~\cite{nonshell}? The drawing of $K_{11}$ shown in
  \figurename~\ref{fig:n11bi} might be a good start.
\item Is there a concept similar to bishellability that does not
  require the starting vertices of the sequences to share a cell, but
  still implies at least $H(n)$ crossings? This would mean that each
  of the two sequences could have their own (local) reference faces.
\end{itemize}

\paragraph*{Acknowledgments.}
Research for this article was initiated during the \emph{AIM Workshop on
Exact Crossing Numbers} held from April 28 to May 2, 2014 in Palo
Alto (California, USA).  We are grateful to the American Institute of
Mathematics (AIM) for their support and we thank the participants of
that workshop for fruitful and inspiring discussions.

Furthermore,
O.A.\ and B.V.\ were partially supported by the ESF EUROCORES programme EuroGIGA -- CRP ComPoSe, Austrian Science Fund (FWF): I648-N18;
S.F.-M.~was supported by the NSF grant DMS-1400653;
B.M.\ was supported in part by an NSERC Discovery Grant R611450 (Canada), by the Canada Research Chairs program, and by the research project J1-8130 of ARRS (Slovenia).
P.M.\ was partially supported by the ESF EUROCORES programme EuroGIGA -- CRP GraDR, 10-EuroGIGA-OP-003 (DFG);
P.R.\ was partially supported by MINECO project MTM2014-54207 and ESF EUROCORES programme EuroGIGA -- CRP ComPoSe, MICINN Project EUI-EURC-2011-4306; and
R.B.R.\ was supported by NSERC grant number 41705-2014 057082.


\newpage

\appendix
\section* {Appendix: Crossings and $k$-edges} \label{App:AppendixA}

If we consider simple drawings of $K_4$, it is well-known that there are only two such drawings up to spherical homeomorphisms.  However, the one with a crossing has, again up to spherical homeomorphisms, two different faces:  one is incident with a 4-cycle and the other is incident with two vertices and the crossing.
It is easy to verify that:
\begin{enumerate}
	\item if $F$ is bounded by the 4-cycle, then the four edges of the 4-cycle are all 0-edges while the crossing edges are both 1-edges;
	\item if $F$ is incident with just two vertices and the crossing, then the two crossing edges are 0-edges, the full edge incident with $F$ is a 0-edge, as is its opposite edge in the uncrossed 4-cycle, and the other two edges in the uncrossed 4-cycle are 1-edges; and
	\item in the case of the planar drawing of $K_4$, the three edges incident with $F$ are the 0-edges, and the other three edges are the 1-edges.
\end{enumerate}

The preceding paragraph can be used to relate crossings in a drawing $D$ of $K_n$ and $k$-edges.  Each crossing determines a $K_4$ containing precisely two 1-edges (relative to the face containing the reference face $F$ of $D$).   Each non-crossing $K_4$ contains three 1-edges.

It follows that if we count the number of ordered pairs $(e,K_4)$ so that $e$ is a 1-edge of this $K_4$, the total number we get is $3P+2N$, where $P$ is the number of planar $K_4$'s and $N$ is the number of non-planar $K_4$'s.  On the other hand, if $uv$ is a $k$-edge, and $w,w'$ are distinct vertices both different from both $u$ and~$v$, then the $K_4$ induced by $u,v,w,w'$ has $uv$ as a 1-edge if and only if one of $w$ and $w'$ is an $R$ (relative to $uv$) and the other is an $L$.  If follows that every $ k $-edge $uv$ is a 1-edge in $k(n-2-k)$ different $K_4$'s.  That is, the number of pairs $(e,K_4)$ where $e$ is a 1-edge of the $K_4$ is
$
\sum_{k=0}^{\lfloor \frac{n}{2}\rfloor-1} k(n-2-k)E_k(D)\,,
$
where, for $k\le \lfloor \frac{n-2}{2}\rfloor = \lfloor \frac{n}{2}\rfloor-1$, $E_k(D)$ denotes the number of $k$-edges in $D$.

The conclusion (exactly as in~\cite{twopage,af,lve}) is that
\begin{equation}\label{eq:1-edge}
 3P+2N=\sum_{k=0}^{\lfloor \frac{n}{2}\rfloor-1} k(n-2-k)E_k(D)\,.
\end{equation}
On the other hand, $P+N$ is the total number of $K_4$'s in $K_n$, so
\begin{equation}\label{eq:numberK4's}
P+N={n\choose 4}\,.
\end{equation}
Multiply Equation~(\ref{eq:numberK4's}) by $3$, subtract Equation~(\ref{eq:1-edge}), and use the obvious fact that $N=\crn(D)$ to conclude that
\begin{equation} 
	\crn(D) = 3{n\choose 4} - \sum_{k=0}^{\lfloor\frac{n}{2}\rfloor-1}k(n-2-k)E_k(D)\,.
\end{equation}
Equivalent to the proof of Proposition~1 in~\cite{twopage} this equation can be rewritten to
\begin{equation}\label{eq:crkedge}
	\begin{array}{rcl}
	\crn(D)& = & 2\sum\limits_{k=0}^{\lfloor\frac{n}{2}\rfloor-2}\sum\limits_{i=0}^{k}\left(  k+1-i\right)  E_{i}(D)
				-\frac{1}{2}\binom{n}{2}\lfloor\frac{n-2}{2}\rfloor \\
			&&  -\frac{1}{2}(1+(-1)^{n})\sum\limits_{i=0}^{\lfloor \frac{n}{2}\rfloor-2}\left(  \lfloor \frac{n}{2}\rfloor-2+1-i\right)  E_{i}(D).
	\end{array}
\end{equation}
If we consider
\[
E_{{\leq}k}\left(  D\right)  :=\sum\limits_{j=0}^{k}E_{j}\left(
D\right)
\]
and
\[
E_{{\leq}{\leq}k}(D)  :=\sum\limits_{j=0}^{k}E_{\leq j}\left(
D\right)  =\sum\limits_{j=0}^{k}\sum\limits_{i=0}^{j}E_{i}\left(
D\right) =\sum\limits_{i=0}^{k}\left(  k+1-i\right)  E_{i}\left(
D\right)
\]
we can get the original formulation in \cite{twopage}:
\begin{equation}\label{eq:twopage_app}
\crn(D)=2\sum\limits_{k=0}^{\lfloor
n/2\rfloor-2}E_{{\leq}{\leq}k}(D)-\frac
{1}{2}\binom{n}{2}\left\lfloor \frac{n-2}{2}\right\rfloor -\frac{1}%
{2}(1+(-1)^{n})E_{\leq\leq\left\lfloor n/2\right\rfloor -2}(D).
\end{equation}

\end{document}